\documentclass[10pt,electronic]{amsart}        
\usepackage{baskervald}
\usepackage[swedish, english]{babel}
\usepackage[T1]{fontenc} 
\usepackage{accents}
\makeatletter
\usepackage{array}
\usepackage{tabularx}

\usepackage{hyperref}
\usepackage[noabbrev,capitalize]{cleveref}
\hypersetup{
	colorlinks,
	linkcolor={red!50!black},
	citecolor={blue!50!black},
	urlcolor={blue!80!black}
}

\usepackage{multicol}
\usepackage{adjustbox}
\usepackage{url}
\usepackage[hmargin=3.4cm,vmargin=3.5cm]{geometry} 
\usepackage{graphicx}                         
\usepackage{amsmath}
\usepackage{amsthm}                      
\usepackage{amssymb} 
\usepackage{mathrsfs}
\usepackage{stmaryrd}
\usepackage{mathtools}
\usepackage[shortlabels]{enumitem}
\usepackage{tikz-cd}
\usetikzlibrary{babel}
\usepackage{euscript}
\renewcommand{\mathcal}{\EuScript}
\theoremstyle{plain}                                              
\newtheorem{thm}{Theorem}[section]

\newtheorem{lem}[thm]{Lemma}
\newtheorem{prop}[thm]{Proposition}
\newtheorem{cor}[thm]{Corollary}

\theoremstyle{definition}
\newtheorem{example}[thm]{Example}
\newtheorem{rem}[thm]{Remark}

\DeclareMathOperator{\Gal}{Gal}
\DeclareMathOperator{\rank}{rank}
\newcommand{\FZ}{\mathsf{FZ}}
\usepackage{caption}
\title{Motivic Euler characteristics of moduli spaces of curves}

\begin{document}

\newcommand{\Q}{\mathbf Q}
\newcommand{\Qbar}{\overline\Q}
\newcommand{\Z}{\mathbf Z}
\newcommand{\Mb}{\overline M}
\renewcommand{\L}{\mathsf L}
\renewcommand{\S}{\mathsf S}
\newcommand{\F}{\mathbf F}
\newcommand{\V}{\mathbb V}
\renewcommand{\P}{\mathbf{P}}

\newcommand{\Dan}[1]{\textcolor{orange}{#1}}
\newcommand{\Jonas}[1]{\textcolor{purple}{#1}}
\newcommand{\Sam}[1]{\textcolor{red}{#1}}

\author{Jonas Bergström}

\author{Samir Canning}

\author{Dan Petersen}

\author{Johannes Schmitt}

\begin{abstract}A combination of recent results of Canning--Larson--Payne--Willwacher and Payne--Willwacher is that the motivic Euler characteristic of the moduli space $M_{g,n}$ of nonsingular curves with $n$ markings is polynomial in the Tate motive if and only if $g=0$ or $3g+2n<25$. In the cases when $g\geq 1$ and $3g+2n<25$, there are $32$ pairs $(g,n)$ where the motivic Euler characteristic has been computed. We compute the motivic Euler characteristic for $9$ more pairs, leaving only $6$ pairs where the motivic Euler characteristic is polynomial but unknown. In genus $4$, we moreover compute the answer when $n=7$, i.e.~the first non-polynomial case.
\end{abstract}

 \maketitle

\section{Introduction}
\subsection{Motivic Euler characteristics}
Let $X$ be a separated $d$-dimensional Deligne--Mumford stack of finite type over $\Q$. The motivic Euler characteristic of $X$ is defined as
\[\chi_c(X) = \sum_{i=0}^{2d} (-1)^i [H^i_c(X_{\Qbar},\Q_{\ell})] \in K_0(\Gal_\Q).\] 
Here $\Gal_{\Q}$ is the category of $\ell$-adic $\Gal(\Qbar/ \Q)$ representations.  If the symmetric group $S_n$ acts on $X$, then $\chi_c^{S_n}(X)$ denotes the evident equivariant refinement $\chi_{c}^{S_n}(X)$ valued in $K_0(\Gal_{\Q}^{S_n})$, the Grothendieck group of the category of $S_n\times \Gal(\Qbar/\Q)$-representations.

 
 Let $\L=[H^2(\P^1,\Q_{\ell})]$. We say that the compactly supported cohomology of $X$ is of Tate type if $[H^i_c(X,\Q_{\ell})]\in K_0(\Gal_\Q)$ is a sum of powers of $\L$ for all $i$.  In this case, $\chi_c(X)$ is a polynomial in $\L$, but the converse does not necessarily hold.  Moreover, if $X$ has a $S_n$-action, then the coefficient of each irreducible representation of $S_n$ in $\chi_c^{S_n}(X)$ is a polynomial in $\L$. 

 The (equivariant) motivic Euler characteristic encodes many natural invariants of $X$. Via the dimension morphism $K_0(\Gal_\Q)\rightarrow \Z$, one obtains the usual (equivariant) Euler characteristic from the motivic one, and via the trace of Frobenius, the (equivariant) motivic Euler characteristic encodes the (equivariant) point counts of $X$ over $\F_q$ for prime powers $q=p^k$, where $p$ is a prime of good reduction for $X$. 

\subsection{Results}
Here, we investigate the $S_n$-equivariant motivic Euler characteristics of moduli spaces of curves. Let $M_{g,n}$ be the moduli space of nonsingular curves of genus $g$ with $n$ markings and $\Mb_{g,n}$ its compactification by stable curves. These moduli spaces are separated nonsingular Deligne--Mumford stacks of finite type over $\Q$ of dimension $3g-3+n$, and they admit an $S_n$-action by permutation of the markings. The following theorem, a combination of work of Canning, Larson, Payne, and Willwacher \cite[Theorem 1.5]{CLPW} and Payne and Willwacher \cite[Theorem 1.1]{PayneWillwacherPoly}, completely classifies when $\chi_c(M_{g,n})$ is polynomial in $\L$.
\begin{thm}\label{thm:CLPW} Let $g$ and $n$ be nonnegative integers such that $2g+n\geq 3$. Then the following are equivalent:
\begin{enumerate}
    \item The motivic Euler characteristic $\chi_c(M_{g,n})$ is polynomial in $\L$;
    \item The compactly supported cohomology of $M_{g,n}$ is of Tate type;
    \item $g=0$ or $3g+2n<25$.
\end{enumerate}

\end{thm}
A natural question is to compute $\chi^{S_n}_c(M_{g,n})$ whenever it is polynomial. When $g=0$, the equivariant motivic Euler characteristic can be computed for all $n$ recursively \cite{GetzlerGenus0}, so we focus on the cases $g\geq 1$ and $3g+2n<25$. For these $47$ pairs $(g,n)$, the polynomials have been previously computed for the $32$ pairs when $g\leq 3$ or $g=4$ and $n\leq 3$ through the combined effort of many authors over many years \cite{getzler-resolving, BergstromFaber, PetersenA2,Petersendelliptic,CLP,BT,Tommasigenus4,BergstromFaberPayne}.\footnote{The cohomology of $M_5$ has been computed by Bergström--Wennink \cite{BergstromWennink}, but the computation of $\chi_c(M_5)$ here was obtained first, as noted in \cite[Remark 1.4]{BergstromWennink}.} 
We compute the polynomials for $9$ of the $15$ remaining pairs. 

Moreover, we also compute $\chi^{S_7}_c(M_{4,7})$, which is \emph{not} polynomial in $\L$. Indeed, it contains $S[12]:=[H^{11}(\Mb_{1,11},\Q_{\ell})]$; the cohomology group $H^{11}(\Mb_{1,11},\Q_{\ell})$ is isomorphic to the irreducible $2$-dimensional Galois representation attached to the Ramanujan cusp form $\Delta$ of weight $12$ for $\mathrm{SL}_2(\Z)$. For any prime $p$, the trace of the $p$th Frobenius on $S[12]$ equals $\tau(p)$, the $p$th value of the Ramanujan tau-function. We remark that in contrast, $\chi^{S_7}_c(\Mb_{4,7})$ is in fact polynomial in $\L$.

\begin{thm}\label{thm:Mgn}We compute $\chi_c^{S_n}(M_{g,n})$ for $g=4$ and $4\leq n\leq 7$, $g=5$ and $n\leq 3$, and $g=6$ and $n\leq 1$. The results when $g=5$ and $g=6$ are as follows: 
{\footnotesize 
\begin{align*}
\chi_c(M_{5})&=\L^{12}+\L^{11}+\L^{10}-\L^8+1
, \\
\chi_c(M_{5,1})&=\L^{13}+2\L^{12}+3\L^{11}+\L^{10}-2\L^9-\L^4+\L+1
, \\
\chi_c^{S_2}(M_{5,2})&=
(\L^{14}+2\L^{13}+4\L^{12}+2\L^{11}-4\L^{10}-\L^9+\L^8-\L^5-\L^4-\L^3+2\L^2+3\L+1)s_{2}\\
&+(\L^{13}+2\L^{12}+2\L^{11}-\L^{10}+\L^8-\L^7-\L^5+\L)s_{1,1}
, \\
\chi_c^{S_3}(M_{5,3})&=(\L^{15}+2\L^{14}+4\L^{13}+2\L^{12}
-6\L^{11}-4\L^{10}+3\L^9-\L^8+2\L^7-\L^6-4\L^5+\L^4+4\L^2+2\L-1)s_{3}\\
&+(\L^{14}+3\L^{13}
+3\L^{12}-4\L^{11}-3\L^{10}+5\L^9-2\L^8+\L^7-\L^6-5\L^5+3\L^4+3\L^2+5\L+1)s_{2,1}\\
&+(\L^{12}+2\L^9-\L^8-\L^7+\L^6+3\L^4-3\L^2)s_{1,1,1}
, \\
\chi_c(M_{6})&=\L^{15}+\L^{14}+2\L^{13}+\L^{12}-\L^{10}+\L^3-1
, \\
\chi_c(M_{6,1})&=
\L^{16}+2\L^{15}+4\L^{14}+4\L^{13}-3\L^{11}+2\L^5+\L^4+2\L^3-\L^2-3\L-1
.
\end{align*}
}
The formulas when $g=4$ 
are tabulated at the end of Section \ref{sec:genus4final}.
\end{thm}

Note that the result when $g=4$ exhausts and even goes beyond the region of polynomiality, but there remain open cases when $5\leq g\leq 8$. 

In contrast to Theorem \ref{thm:CLPW}, there is not yet a complete understanding of when the cohomology of $\Mb_{g,n}$ is of Tate type, although there are only finitely many pairs of $(g,n)$ for which it is not known if the cohomology of $\Mb_{g,n}$ is of Tate type \cite[Theorem 1.5]{CLPWFA}, \cite[Theorem 1.7]{CLPW}. 
\begin{rem} \label{smaller Eulers}
 The equivariant motivic Euler characteristic $\chi_c^{S_n}(\Mb_{0,n})$ is always a polynomial in $\L$ and can be computed 
 for instance using the Getzler--Kapranov formula. 
 In genus $1$, $\chi^{S_n}_c(\Mb_{1,n})$ is a polynomial in $\L$ if $n\leq 10$, and the polynomial has been computed. The same holds for $g=2$ and $n\leq 9$, $g=3$ and $n\leq 9$, and $g=4$ and $n\leq 3$. The outcomes can be found in the Github repository~\cite{CohomologyModuliSpaces}. \end{rem}
We extend these computations of $\chi_c^{S_n}(\Mb_{g,n})$ to the same values of $g$ and $n$ as in Theorem \ref{thm:Mgn}.

\begin{thm}\label{thm:Mbargn}
We compute $\chi_c^{S_n}(\Mb_{g,n})$ for $g=4$ and $4\leq n\leq 7$, $g=5$ and $n\leq 3$, and $g=6$ and $n\leq 1$. In these cases, $\chi^{S_n}_c(\Mb_{g,n})$ is a polynomial in $\L$.
\end{thm}
The actual formulas are displayed in Section \ref{sec:gatleast5}, where the proof is given for the cases with $g\geq 5$, and in Section \ref{sec:genus4final} for the $g=4$ cases. They are also available in the Github repository \cite{CohomologyModuliSpaces}. 
\begin{rem}If $X$ is smooth and proper, then each Galois representation $H^i_c(X,\Q_{\ell})$ is pure of weight $i$, and hence $\chi_c^{S_n}(X)$ determines each  $H^i_c(X,\Q_{\ell})$ as an element of $\Gal^{S_n}_\Q$ up to semisimplification. Therefore, Theorem \ref{thm:Mbargn} determines the Betti numbers of $\Mb_{g,n}$ in the relevant cases. On the other hand, Theorem \ref{thm:Mgn} does not determine the Betti numbers of $M_{g,n}$ in those cases. See Remark \ref{rem:categorify} for a discussion of the extra information needed to compute the cohomology of $M_{g,n}$. \end{rem}
\subsection{Methods}


Consider the stratification of $\overline M_{g,n}$ by topological type. All strata are products of smaller moduli spaces $M_{g',n'}$ modulo a finite group. By ``smaller'' we mean here that $g' \leq g$, and $2g'-2+n' \leq 2g-2+n$. In particular, it follows that if you know $\chi_c^{S_{n'}}(M_{g',n'})$ for all such $g'$ and $n'$, then you can compute $\chi_c^{S_n}(\overline M_{g,n})$. Moreover, the quantities $\{\chi_c^{S_n}(M_{g,n})\}_{g,n}$ and $\{\chi_c^{S_n}(\overline M_{g,n})\}_{g,n}$ are related by a triangular matrix with ones along the diagonal, so that one may conversely reconstruct the motivic Euler characteristics of the uncompactified moduli spaces from those of the Deligne--Mumford compactifications.    

 If one wanted to implement this observation on a computer, one would need to write a program that enumerates all stable graphs of given genus and number of legs and their automorphism groups. A more systematic approach to this combinatorial problem was proposed by Getzler and Kapranov \cite{gk}. They write down generating series --- let's call them $A$ and $B$ --- which are the sum of all $S_n$-equivariant motivic Euler characteristics of all the moduli spaces $M_{g,n}$ (resp. $\overline M_{g,n}$). Then they give a simple and explicit operator in terms of symmetric functions that transforms $A$ into $B$, and vice versa; their formula has the effect of carrying out the above sum over stable graphs and taking coinvariants of their automorphism groups. Symmetric functions arise here as a convenient way of encoding the representation theory of the symmetric groups in characteristic zero. See Section \ref{sec:GetzlerKapranov} for more details.
 
 \begin{rem}\label{rem:categorify}
 The spectral sequence induced by the stratification by topological type and the weight spectral sequence for the normal crossings pair $(\Mb_{g,n}, \partial M_{g,n})$ lift the above observation about the relationship between Euler characteristics of $\Mb_{g,n}$ and $M_{g,n}$ to a relationship between cohomology groups. It is much more difficult to compute cohomology than Euler characteristics because computing cohomology involves studying all of the arrows of the spectral sequences, rather than just the entries. Recently, Bergström and Wennink have implemented the weight spectral sequence in the Sagemath package \texttt{admcycles} \cite{admcycles} and used it to compute the cohomology of $M_5$ \cite{BergstromWennink}.
 \end{rem}

The first applications of the Getzler--Kapranov formula computed $\chi_c^{S_n}(\Mb_{g,n})$ for small $g$ and $n$ by first computing $\chi_c^{S_n}(M_{g',n'})$ for ``smaller" $g',n'$ \cite{GetzlerGenus0,semiclassical,GL,BT}. In more recent works, the flow of information often goes both ways; see for example \cite{BergstromFaberPayne,bcgp}. In this paper, too, we take the reversed perspective, and use the Getzler--Kapranov formula to compute the motivic Euler characteristic of $M_{g,n}$ from data about Deligne--Mumford compactifications. That is, we use Theorem~\ref{thm:Mbargn} and the Getzler--Kapranov formula to prove Theorem \ref{thm:Mgn}. What makes this possible is a recent result of Canning--Larson \cite{CL}, which implies that in a significant range of values of $g$ and $n$, computing the cohomology of $\overline M_{g,n}$ is in fact \emph{much easier} than computing the cohomology of $M_{g,n}$.  

Let $RH^{\bullet}(\Mb_{g,n})$ denote the tautological subring of $H^{\bullet}(\Mb_{g,n})$ (see Section \ref{sec:tautological} for definitions).
\begin{thm}\label{thm:CL} If $g \geq 3$ and $2g-2+n \leq 12$, then $H^\bullet(\overline M_{g,n},\Q) =  RH^{\bullet}(\overline M_{g,n})$. The same holds for $g=2$ and $n \leq 9$, $g=1$ and $n \leq 10$, and $g=0$ and any $n$. 
\end{thm}
The genus $0$ result is due to Keel \cite{Keel} and the genus $1$ result is due to  Belorousski \cite{Belorousski}. Several other cases for small values of $g$ and $n$ were also known previously \cite{Mumford, Faber1, Faber2}. By definition, the tautological ring is generated by algebraic cycles, and so $H^{\bullet}(\Mb_{g,n})$ is of Tate type for the pairs $(g,n)$ in Theorem \ref{thm:CL}. The computation of $\chi_c(\Mb_{g,n})$ thus reduces to a computation of the tautological ring.

The tautological ring $RH^{2\bullet}(\Mb_{g,n})$ has an explicit set of generators called decorated strata, which span a graded $\Q$-algebra called the strata algebra $\mathsf S^\bullet(\Mb_{g,n})$. The algebra multiplication is defined formally to agree with that of $RH^{2\bullet}(\Mb_{g,n})$, so there is map of graded $\Q$-algebras 
\[
\S^\bullet(\Mb_{g,n})\rightarrow RH^{2\bullet}(\Mb_{g,n}).
\]
This map factors through the tautological Chow ring $R^\bullet(\Mb_{g,n})$, and
 Pixton, see \cite{PPZ}, has conjectured a generating set for the ideal of relations $\mathsf{I}^{\bullet}(\Mb_{g,n})$, the kernel of the map to $R^\bullet(\Mb_{g,n})$. We call the ideal generated by Pixton's relations $\FZ^{\bullet}(\Mb_{g,n})$.\footnote{Pixton's relations are also known as the generalized Faber--Zagier relations, hence the notation.} Pandharipande, Pixton, and Zvonkine \cite{PPZ} proved that $\FZ^{\bullet}(\Mb_{g,n})$ maps to zero in $RH^{\bullet}(\Mb_{g,n})$.
 
 We do not discuss the details of Pixton's relations $\FZ^{\bullet}(\Mb_{g,n})$ here. Instead, we note that the strata algebra $\S^{\bullet}(\Mb_{g,n})$ and Pixton's relations $\FZ^{\bullet}(\Mb_{g,n})$ are implemented in the software package \texttt{admcycles} \cite{admcycles}. 
For a fixed $(g,n,i)$, we can thus use \texttt{admcycles} to obtain an upper bound for $\dim RH^{2i}(\Mb_{g,n})$. We also use \texttt{admcycles} to obtain a lower bound via the intersection pairings
\[
\S^i(\Mb_{g,n})\times \S^{3g-3+n-i}(\Mb_{g,n})\rightarrow RH^{6g-6+2n}(\Mb_{g,n})\cong \Q.
\]
If the upper and lower bounds agree, then we have successfully computed $[RH^{2i}(\Mb_{g,n})]\in K_0(\Gal_\Q)$. Moreover, we can even calculate $[RH^{2i}(\Mb_{g,n})]\in K_0(\Gal_{\Q}^{S_n})$, as \texttt{admcycles} allows us to calculate $S_n$-equivariantly. See Section \ref{sec:admcycles} for details. Doing these computations becomes quite difficult when $g$, $n$, and $i$ increase, as the size of the strata algebra and the set of relations grows significantly.

To avoid these more difficult computations, we use known results about the Euler characteristic and point counts over small finite fields to impose linear relations among the coefficients of $\chi^{S_n}_c(\Mb_{g,n})$ \cite{Dragutinovic, KHL}. 

Combining these methods suffices to prove Theorem \ref{thm:Mbargn}. Using the Getzler--Kapranov method, we then prove Theorem \ref{thm:Mgn}.

\subsection{Beyond Theorems \ref{thm:Mgn} and \ref{thm:Mbargn}}
In principle, the methods of this paper could compute $\chi_c^{S_n}(M_{g,n})$ for the $6$ remaining polynomial pairs $(g,n)$. In practice, there are computational obstructions to doing so because of the size of the strata algebra $\S^i(\Mb_{g,n})$ for $i$ large. We have made some partial calculations. For example, using the method briefly explained above and fully explained in Section \ref{sec:admcycles}, we have computed every term of the polynomial $\chi_c(\Mb_7)$ except for the terms of degrees $8$, $9$, and $10$:
\[\chi_c(\Mb_7)=1 + 5 \L + 32 \L^2 + 162 \L^3 + 695 \L^4 + 2388 \L^5 +  6498 \L^6 +  13539 \L^7 + ? \L^8 + ?\L^9 + ?\L^{10} +\dots\]
Note that the coefficient of degree $8$ and coefficient of degree $10$ are the same by Poincar\'e duality, and the elided terms in the above expression are also determined by Poincar\'e duality, so there are only two remaining unknowns. The Euler characteristic provides one linear relation among the coefficients, but at present there is no point counting data to provide another linear relation. We have similar calculations for other unknown cases, such as $\chi_c^{S_4}(\Mb_{5,4})$ and $\chi_c^{S_2}(\Mb_{6,2})$. If we could compute these three motivic Euler characteristics, then we could also compute $\chi_c^{S_4}(M_{5,4})$, $\chi_c^{S_2}(M_{6,2})$, and $\chi_c(M_7)$.

Computational improvements in the methods of Section \ref{sec:admcycles} and more point counting data over small finite fields would both help to push these results further.

\subsection*{Acknowledgments}
We thank Carel Faber for sharing his point counting data in genus $4$ with us. We also thank Steve Huang, Kiran Kedlaya, Jun Bo Lau, Sam Payne, and Thomas Willwacher for helpful conversations. The computer calculations in this paper were carried out on the servers of ETH Z\"urich. We thank the IT support group for helping to facilitate these calculations.

\subsection*{AI Declaration}
Artificial intelligence was not used for the mathematics, writing, computation, or any other part of this paper.

\section{$S_n$-equivariant motivic Euler characteristics}
We discuss some abstract properties of $S_n$-equivariant motivic Euler characteristics and the Getzler--Kapranov method.



\subsection{Symmetric functions}\label{section: symmetric functions} On the level of Grothendieck groups, there is a decomposition \cite[Theorem 3.2]{getzler-mixedhodge}
\[ K_0(\Gal^{S_n}_\Q) \cong K_0(\mathrm{Rep}(S_n)) \otimes_\Z K_0(\Gal_\Q)  \cong  \Lambda_n \otimes_\Z K_0(\Gal_\Q),\]
where $\Lambda_n$ is the degree $n$ component of the ring of symmetric functions. We write $s_\lambda$ for the \emph{Schur polynomial} in $\Lambda_n$ associated to the partition $\lambda$ of $n$ and $p_i$ the $i$th \emph{power sum}. Under the isomorphism $K_0(\mathrm{Rep}(S_n)) \cong \Lambda_n$, it corresponds to the class of the representation $\sigma_\lambda$ of $S_n$ associated to the partition $\lambda$. For each $i$, we have a decomposition into isotypical components
\[
H^i_c(X,\Q_{\ell}) = \bigoplus_{\lambda \vdash n} \sigma_{\lambda}\boxtimes V_{i,\lambda},
\]
and there is an equality
\[ \chi_c^{S_n}(X) = \sum_{i=0}^{2d} (-1)^i \sum_{\lambda \vdash n} s_\lambda \otimes [V_{i,\lambda}].\]
We let $\langle-,-\rangle$ be the standard inner product on $\Lambda_n$, for which the Schur polynomials form an orthonormal basis \cite[Section I.4]{macdonald}. Hence
\[ \langle \chi_c^{S_n}(X), s_\lambda\rangle = \sum_{i=0}^{2d} (-1)^i [V_{i,\lambda}].\]
The equivariant motivic Euler characteristic can be expressed in terms of the non-equivariant one as follows. If $\lambda$ is a partition of $n$, then we write $S_\lambda$ for the associated Young subgroup of $S_n$. Explicitly, $S_\lambda = S_{\lambda_1} \times S_{\lambda_2} \times \dots \times S_{\lambda_k}$, where $k=l(\lambda)$. 
\begin{lem}\label{lem:Young}
Let $m_\lambda$ be the monomial symmetric polynomial associated to the partition $\lambda$. Then 
\[ \chi_c^{S_n}(X) = \sum_{\lambda \vdash n} m_\lambda \otimes \chi_c(X/S_\lambda).\]

\end{lem}

\begin{proof}Let $h_\lambda$ be the complete homogeneous symmetric polynomial associated to the partition $\lambda$. Since the Schur polynomials are an orthonormal basis with respect to $\langle-,-\rangle$, Schur's lemma implies that the inner product satisfies 
$$ \langle [V], [W] \rangle = \dim \mathrm{Hom}_{S_n}(V,W)$$
for $V,W \in \mathrm{Rep}(S_n)$. It then follows by Frobenius reciprocity that 
\[ \langle h_\lambda, [W] \rangle = \dim \left(W^{S_\lambda}\right),\]
since $h_\lambda = [\operatorname{Ind}_{S_\lambda}^{S_n} \Q]$, where $\Q$ denotes the trivial one-dimensional representation of $S_\lambda$.  
    Hence $\langle h_\lambda, [H^i_c(X,\Q_\ell)]\rangle = \dim H^i_c(X/S_\lambda,\Q_\ell)$ and $\langle h_\lambda,\chi_c^{S_n}(X) \rangle = \chi_c(X/S_\lambda)$. The result follows from this fact, and that $m_\lambda$ and $h_\lambda$ are dual bases of $\Lambda_n$ with respect to $\langle-,-\rangle$  \cite[Section I.4]{macdonald}.
\end{proof}

\subsection{Tate motives} Recall that $\L=[H^2(\mathbb{P}^1,\Q_{\ell})]\in K_0(\Gal_\Q)$. 
Here, we are focused on stacks $X$ whose compactly supported cohomology is of Tate type; that is, stacks $X$ such that  $[H^i_c(X,\Q_{\ell})]$ is a sum of powers of $\L$ for all $i$. Then also each coefficient of $s_{\lambda}$ in $\chi_c^{S_n}(X)$ is a polynomial in $\L$, when $X$ carries an action of $S_n$.

\begin{lem}\label{lem:specializations}
    Suppose that $X$ is such that $\chi_c(X) = f(\L)$ is a polynomial in $\L$. Then
    \begin{enumerate}
        \item The Euler characteristic $e(X)$ is given by $f(1)$.
        \item Let $p$ be a prime, and $t$ a positive integer. If $X$ is defined over $\Z_{(p)}$ and has good reduction mod $p$, then  the number $\# X(\mathbf F_{p^t})$ of $\mathbf F_{p^t}$-points of $X$ is given by $f(p^t)$.
    \end{enumerate}
\end{lem}

\begin{proof}There is a ring homomorphism $K_0(\Gal_\Q) \to \Z$ defined by sending an $\ell$-adic Galois representation to its dimension. This homomorphism maps $\L$ to $1$, and it also maps $\chi_c(X)$ to the compactly supported Euler characteristic of $X$, by definition. This coincides with the ordinary Euler characeristic by \cite{laumon}.

    Let $\Gal'_\Q$ denote the category of $\ell$-adic Galois representations unramified at $p$. There is a similar homomorphism $K_0(\Gal'_\Q) \to \Z$ which sends an $\ell$-adic Galois representation $V$ to $\operatorname{Tr}(\operatorname{Frob}_{p^t} |\, V)$. It maps $\L$ to $p^t$, and sends $\chi_c(X)$ to the number of $\mathbf F_{p^t}$-points, using the Grothendieck--Lefschetz trace formula and in addition the good reduction hypothesis to identify \'etale cohomology over $\Qbar$ and $\overline{\mathbf{F}}_p$.
\end{proof}

\subsection{The Getzler--Kapranov method}\label{sec:GetzlerKapranov}

A systematic approach to passing between the two sets of data
\begin{equation}\label{eq:correspondenceGK}
    \{ \langle s_\lambda,\chi_c^{S_n}(\overline M_{g,n})\rangle \}  \longleftrightarrow  
    \{\langle s_\lambda,\chi_c^{S_n}(M_{g,n})\rangle\} \end{equation} 
was developed by Getzler and Kapranov \cite{gk}. 
Let
 $$ A = \sum_{g,n} \hbar^{g-1} \chi_c^{S_n}(M_{g,n}) \qquad \text{and} \qquad B = \sum_{g,n} \hbar^{g-1} \chi_c^{S_n}(\overline M_{g,n}), $$
 where $\hbar$ is an indeterminate. 
 The \emph{Getzler--Kapranov formula} reveals a striking symmetry in the situation:
 $$ A = \mathbf{Log}\big( \exp(\boldsymbol{\Delta})\, \mathbf{Exp}(B)\big ) \qquad \iff \qquad B = \mathbf{Log}\big( \exp(-\boldsymbol{\Delta})\, \mathbf{Exp}(A)\big ). $$
 Here $\mathbf{Exp}$, $\mathbf{Log}$, and $\boldsymbol\Delta$ are the ``plethystic exponential'', ``plethystic logarithm'', and ``plethystic Laplacian'', respectively. These are universal operations acting on complete $\lambda$-rings, and we are using that both generating series $A$ and $B$ live in a completion of $\Lambda \otimes K_0(\Gal_\Q)$, where $\Lambda = \bigoplus \Lambda_n$ is the ring of symmetric functions, and that $\Lambda \otimes K_0(\Gal_\Q)$ comes with a natural $\lambda$-ring structure. From a highbrow perspective, the symmetry in the Getzler--Kapranov formula is a reflection of the self-duality of the Feynman transform.
 
The details of the formula are explained in \cite{gk}, but what is most important to us are the following two facts. First, the formula exists and can be efficiently implemented using e.g. SageMath's methods for computations with symmetric functions. Second, for a fixed pair $(g,n)$ computing $\chi_c^{S_n}(\Mb_{g,n})$ (respectively, $\chi_c^{S_n}(M_{g,n})$) requires only the data of $\chi_c^{S_{n'}}(M_{g',n'})$ (respectively, $\chi_c^{S_{n'}}(\Mb_{g',n'})$) for $0<2g'-2+n'\leq 2g-2+n$ and $g'\leq g$. With these facts, we can prove Theorem \ref{thm:Mgn}, assuming Theorem \ref{thm:Mbargn}.
 \begin{proof}[Proof of Theorem \ref{thm:Mgn}, assuming Theorem \ref{thm:Mbargn}]
     For any pair $(g,n)$ in the statement, $\chi^{S_n}_c(\Mb_{g',n'})$ has been computed for all $g',n'$ such that $0<2g'-2+n'\leq 2g-2+n$ and $g'\leq g$ by Theorem \ref{thm:Mbargn} and previous calculations of several authors, which can be found at \cite{CohomologyModuliSpaces}. Applying the Getzler--Kapranov method yields the result.
 \end{proof}

\section{Cohomology of $\Mb_{g,n}$}
It remains to prove Theorem \ref{thm:Mbargn}. Here, we explain the methods to do so when $g\geq 5$. The $g=4$ case uses different methods, explained in Section \ref{sec:genus4}.
\subsection{The tautological ring}\label{sec:tautological}
For a stable graph $\Gamma$ of genus $g$ with $n$ legs and vertex set $V$, we set \[\Mb_{\Gamma}=\prod_{v\in V} \Mb_{g(v),n(v)}.\]
There is a proper gluing morphism
\[
\xi_{\Gamma}:\Mb_{\Gamma}\rightarrow \Mb_{g,n}
\]
whose image is the closure of the locus of curves with dual graph $\Gamma$. Let $\pi:\Mb_{g,n+1}\rightarrow \Mb_{g,n}$ be the universal curve and $s_i:\Mb_{g,n}\rightarrow \Mb_{g,n+1}$ the universal sections. The $\psi$ classes on $\Mb_{g,n}$ are defined as
\[
\psi_i=c_1(s_i^*\omega_{\pi})\in \mathsf{A}^1(\Mb_{g,n}).
\]
The Arbarello--Cornalba $\kappa$ classes are 
\[
\kappa_j = \pi_*(\psi_{n+1}^{j+1})\in \mathsf{A}^j(\Mb_{g,n}).
\]
Here, $\mathsf{A}^*(\Mb_{g,n})$ denotes the Chow ring of $\Mb_{g,n}$ with rational coefficients.
Let $\pi_v:\Mb_{\Gamma}\rightarrow \Mb_{g(v),n(v)}$ be the projection map. A \emph{decoration} on $\Gamma$ is a product of monomials in $\psi$ and $\kappa$ classes of total degree at most $3g(v)-3+n(v)$ pulled back from the moduli spaces associated to the vertices by $\pi_v$. Let $\S^\bullet(\Mb_{g,n})$ be the $\Q$-vector space whose basis elements are pairs $[\Gamma,\gamma]$, where $\Gamma$ is a stable graph of genus $g$ with $n$ legs and $\gamma$ is a decoration. The vector space $\S^\bullet(\Mb_{g,n})$ is finite dimensional and graded:
\[
\deg\ [\Gamma,\gamma] = |E(\Gamma)|+\deg(\gamma),
\]
where $E(\Gamma)$ is the set of edges of $\Gamma$.
It also has the structure of a $\Q$-algebra, with the product determined by the intersection theory of the boundary strata of $\Mb_{g,n}$, see \cite[Appendix]{GraberPandharipande} for details. The product respects the grading, and so $\S^\bullet(\Mb_{g,n})$ is a graded algebra called the \emph{strata algebra}. Elements in the strata algebra give rise to Chow classes by pushforward under gluing morphisms $\Mb_{\Gamma}\rightarrow \Mb_{g,n}$ and cohomology classes via the degree doubling cycle class map:
\begin{equation}\label{eq:stratamap}
\S^\bullet(\Mb_{g,n})\rightarrow \mathsf{A}^{\bullet}(\Mb_{g,n})\rightarrow H^{2\bullet}(\Mb_{g,n}).
\end{equation}
 The image of \eqref{eq:stratamap} is by definition the (cohomological) tautological ring $RH^{2\bullet}(\Mb_{g,n})$. The kernel of \eqref{eq:stratamap} is called the ideal of tautological relations $\mathsf{I}^{\bullet}(\Mb_{g,n})$. Pixton \cite[Conjecture 5.2]{PixtonPHD} has conjectured a generating set for the ideal of relations $\mathsf{I}^{\bullet}(\Mb_{g,n})$. We call the ideal generated by Pixton's relations $\FZ^{\bullet}(\Mb_{g,n})$. Pandharipande, Pixton, and Zvonkine \cite{PPZ} proved that $\FZ^{\bullet}(\Mb_{g,n})\subset \mathsf{I}^{\bullet}(\Mb_{g,n})$, and Pixton's conjecture is that this containment is in fact an equality.

\subsection{Calculating $RH^{2i}(\Mb_{g,n})$ using \texttt{admcycles}}\label{sec:admcycles}
We do not discuss the details of Pixton's relations $\FZ^{\bullet}(\Mb_{g,n})$ here. Instead, we note that the strata algebra $\S^{\bullet}(\Mb_{g,n})$ and Pixton's relations $\FZ^{\bullet}(\Mb_{g,n})$ are implemented in the SageMath software package \texttt{admcycles} \cite{admcycles}. 
For fixed $(g,n,i)$, we can use \texttt{admcycles} to obtain an upper bound for $\dim RH^{2i}(\Mb_{g,n})$. First, the program enumerates a basis for $\S^{i}(\Mb_{g,n})$. Then, it creates a matrix $A^i_{g,n}$ whose columns correspond to the basis vectors for $\S^{i}(\Mb_{g,n})$ and whose rows correspond to a generating set of $\FZ^{i}(\Mb_{g,n})$. Then 
\[
\#\mathrm{col}(A^i_{g,n})-\rank(A^i_{g,n}) \geq \dim RH^{2i}(\Mb_{g,n}),
\]
and equality holds if Pixton's conjecture is true.

In practice, the matrix $A^i_{g,n}$ can be very large,\footnote{For the calculations done in this paper, the size can be on the order of hundreds of thousands of rows and columns.} and so it can be difficult to compute the rank. It is easier to reduce the matrix $A^i_{g,n}$ modulo $p$ for some prime $p$ and then compute the rank of the reduced matrix. This reduction is possible so long as $p$ does not divide the denominators of any entry of $A^i_{g,n}$. We denote the reduced matrix by $A^i_{g,n}(p)$. We have
\[
\rank(A^i_{g,n})\geq \rank(A^i_{g,n}(p)),
\]
and so 
\[
\#\mathrm{col}(A^i_{g,n})-\rank(A^i_{g,n}(p)) \geq \dim RH^{2i}(\Mb_{g,n}).
\]
Oftentimes, we compute $\rank(A^i_{g,n}(p))$ by computing a row echelon form. The computation of the row echelon form yields a generating set $\mathsf{G}^i(\Mb_{g,n})$ for $RH^{2i}(\Mb_{g,n})$ that is much smaller than $\S^i(\Mb_{g,n})$.

The rank of the intersection pairing
\[
\S^i(\Mb_{g,n})\times \S^{3g-3+n-i}(\Mb_{g,n})\rightarrow RH^{6g-6+2n}(\Mb_{g,n})\cong \Q
\]
gives a lower bound for $\dim RH^{2i}(\Mb_{g,n})$. The pairing is implemented in \texttt{admcycles}. In practice, we do not compute the full intersection pairing. Instead, we use the generating set $\mathsf{G}^i(\Mb_{g,n})\subset \S^i(\Mb_{g,n})$ and compute the rank of the pairing
\[
\mathsf{G}^i(\Mb_{g,n})\times \S^{3g-3+n-i}(\Mb_{g,n})\rightarrow RH^{6g-6+2n}(\Mb_{g,n})\cong \Q.
\]
If the upper and lower bounds agree, then we have calculated $\dim RH^{2i}(\Mb_{g,n})$.

\begin{rem}
    The above calculations can also be done $S_{\lambda}$-equivariantly, where $S_{\lambda}\subset S_n$ is a Young subgroup as in Section \ref{section: symmetric functions}. Applying Lemma \ref{lem:Young}, we can thus compute $[RH^i(\Mb_{g,n})]\in K_0(\Gal_\Q^{S_n})$.
\end{rem}

\subsection{Proof of Theorem \ref{thm:Mbargn} when $g\geq 5$}\label{sec:gatleast5}
Here, we compute $\chi_c^{S_n}(\Mb_{g,n})$ for the pairs $(g,n)$ in Theorem~\ref{thm:Mbargn} where $g\geq 5$. Each case is presented as a separate example.

\begin{example}\label{ex:Mb5}
We calculate $\chi_c(\Mb_5)$. By Theorem \ref{thm:CL}, we have 
\[
\chi_c(\Mb_5) = \sum_{i=0}^{12} \dim RH^{2i}(\Mb_5)\L^i.
\]
By Poincar\'e duality, it suffices to compute 
\[
\sum_{i=0}^{6} \dim RH^{2i}(\Mb_5)\L^i.
\]
Using the method outlined in Section \ref{sec:admcycles}, we compute
\[
\sum_{i=0}^{5} \dim RH^{2i}(\Mb_5)\L^i = 1+ 4\L + 19\L^2 + 65\L^3 + 165\L^4 + 291\L^5
\]
To compute the remaining coefficient, we use Lemma \ref{lem:specializations}. The numerical Euler characteristic $e(\Mb_5)$ determines a linear relation among the coefficients of $\chi_c(\Mb_5)$. We calculate $e(\Mb_5)=1445$ using \cite{Gorsky} and the Getzler--Kapranov method as in Section \ref{sec:GetzlerKapranov} (one could alternatively have used \cite{biniharer}), and we thus determine $\dim RH^{12}(\Mb_5)=355$.
In total, we have
\[
\chi_c(\Mb_5)=\L^{12}+4\L^{11}+19\L^{10}+65\L^9+165\L^8+291\L^7+355\L^6
+291\L^5+165\L^4+65\L^3+19\L^2+4\L+1
\]
\end{example}

\begin{example}\label{ex:Mb51}
    We calculate $\chi_c(\Mb_{5,1})$. The method is the same as for $\chi_c(\Mb_{5})$. 
    By Theorem \ref{thm:CL}, we have 
\[
\chi_c(\Mb_{5,1}) = \sum_{i=0}^{13} \dim RH^{2i}(\Mb_{5,1})\L^i.
\]
By Poincar\'e duality, it suffices to compute 
\[
\sum_{i=0}^{6} \dim RH^{2i}(\Mb_{5,1})\L^i.
\]
Using the method outlined in Section \ref{sec:admcycles}, we compute
\[
\sum_{i=0}^{5} \dim RH^{2i}(\Mb_{5,1})\L^i = 1 + 7\L + 44\L^2 + 197\L^3 + 624\L^4 + 1371\L^5.
\]
Using the Getzler--Kapranov formula and \cite{Gorsky}, we see $e(\Mb_{5,1})=8578$. Thus, we have $\dim RH^{12}(\Mb_{5,1})=2045$. In total, we have

\begin{equation*}
    \begin{split}
\chi_c(\Mb_{5,1})&=\L^{13}+7\L^{12}+44\L^{11}+197\L^{10}+624\L^9+1371\L^8+2045\L^7
+2045\L^6+1371\L^5+\\&624\L^4+197\L^3+44\L^2+7\L+1.
    \end{split}
\end{equation*}
\end{example}

\begin{example}\label{ex:Mb52}
    We calculate $\chi_c^{S_2}(\Mb_{5,2})$. There are two new ingredients when compared with Examples~\ref{ex:Mb5} and \ref{ex:Mb51}. The first is that we use Lemma \ref{lem:Young} and Poincar\'e duality to reduce the problem to the calculation of 
    \[\sum_{i=0}^{7} \dim RH^{2i}(\Mb_{5,2})\L^i \quad \text{ and } \quad \sum_{i=0}^{7} \dim RH^{2i}(\Mb_{5,2})^{S_2}\L^i.\]
    Using the method of Section \ref{sec:admcycles}, we calculate
    \[
    \sum_{i=0}^5 RH^{2i}(\Mb_{5,2}) \L^i = 1 + 13\L + 113\L^2 + 658\L^3 + 2627\L^4+7147\L^5
    \]
    and
    \[
    \sum_{i=0}^{5} RH^{2i}(\Mb_{5,2})^{S_2} \L^i = 1 + 10\L + 81\L^2 + 443\L^3 + 1714\L^4 + 4566\L^5.
    \]
    
    The second new ingredient is that we use Lemma \ref{lem:specializations} to obtain \emph{two} linear relations among the coefficients of the two Poincar\'e polynomials. The first relation comes from the Euler characteristics $e(\Mb_{5,2})=63501$ and $e(\Mb_{5,2}/S_2) =40398$, see \cite{Gorsky}.
    The second relation comes from the calculation of $\#\Mb_{5,2}(\F_{2})=14824959$ and $\#(\Mb_{5,2}/S_2)(\F_2)=9565841$ using the Getzler--Kapranov method and Dragutinovi\'c's calculation of $\#M_{5,2}(\F_2)$ and $\#(M_{5,2}/S_2)(\F_2)$ \cite{Dragutinovic}. The outcome is that 
    \begin{equation*}
    \begin{split}
    \chi_c^{S_2}(\Mb_{5,2})&=(\L^{14}+10\L^{13}+81\L^{12}+443\L^{11}+1714\L^{10}+4566\L^9+8310\L^8+10148\L^7
+8310\L^6+\\&4566\L^5+1714\L^4+443\L^3+81\L^2+10\L+1)s_2+
(3\L^{13}+32\L^{12}+215\L^{11}+913\L^{10}+\\&2581\L^9+4828\L^8+5959\L^7
+4828\L^6+2581\L^5+913\L^4+215\L^3+32\L^2+3\L)s_{1,1}.
\end{split}
   \end{equation*}
\end{example}
\begin{example}\label{ex:Mb53}
We calculate $\chi_c^{S_3}(\Mb_{5,3})$. The method is the same as Example \ref{ex:Mb52}. Using \texttt{admcycles}, we calculate

    \[\sum_{i=0}^{5} \dim RH^{2i}(\Mb_{5,3})^{S_3} \L^i = 1 + 13 \L + 126 \L^2 + 833 \L^3 + 3833 \L^4 + 12240 \L^5, \]

\[\sum_{i=0}^{5} \dim RH^{2i}(\Mb_{5,3})^{S_2}\L^i = 1 + 19 \L + 213 \L^2 + 1537 \L^3 + 7478 \L^4 + 24713 \L^5,\]
and
\[
\sum_{i=0}^{5}  \dim RH^{2i}(\Mb_{5,3}) \L^i = 1 + 25 \L + 307 \L^2 + 2340 \L^3 + 11765 \L^4 + 39663 \L^5.
\]
We then obtain two more linear relations among the coefficients using the Euler characteristics $e(\Mb_{5,3})^{S_3}=167424$, $e(\Mb_{5,3})^{S_2}=344218$, and $e(\Mb_{5,3})=558290$, see \cite{Gorsky}, and the point counts $\#(\Mb_{5,3}/S_3)(\F_{2}) =56282859$ and $\#(\Mb_{5,3}/S_2)(\F_2)=113588355$, and $\#\Mb_{5,3}(\F_{2})=182374671$, which are again obtained from the Getzler--Kapranov method and Dragutinović's equivariant point counts for $M_{5,3}$. The outcome is that
    \begin{equation*}
    \begin{split}
    \chi_c^{S_3}(\Mb_{5,3})&=(\L^{15}+13\L^{14}+126\L^{13}
+833\L^{12}+3833\L^{11}+12240\L^{10}+26819\L^9+39847\L^8+\\&39847\L^7+26819\L^6
+12240\L^5+3833\L^4+833\L^3+126\L^2+13\L+1)s_3+\\&(6\L^{14}+87\L^{13}+704\L^{12}
+3645\L^{11}+12473\L^{10}+28443\L^9+43039\L^8+43039\L^7+\\&28443\L^6+12473\L^5
+3645\L^4+704\L^3+87\L^2+6\L)s_{2,1}+\\&(7\L^{13}+99\L^{12}+642\L^{11}+2477\L^{10}
+6028\L^9+9386\L^8+9386\L^7+6028\L^6+2477\L^5+\\&642\L^4+99\L^3+7\L^2)s_{1,1,1}.
\end{split}
   \end{equation*}

\end{example}

\begin{rem}
    Dragutinović's  point counts over $\F_2$ could have been used in Examples \ref{ex:Mb5} and \ref{ex:Mb51} as well, but the computer calculations we did were sufficiently quick so as to not need the extra linear relation among the coefficients.
\end{rem}
\begin{example}\label{ex:Mb6}
We calculate $\chi_c(\Mb_{6})$. The method is the same as in Examples \ref{ex:Mb5} and \ref{ex:Mb51}.
Using \texttt{admcycles}, we calculate
\[
\sum_{i=0}^{6} \dim RH^{2i}(\Mb_6) \L^{i} = 1 + 5\L + 25 \L^2 + 111 \L^3 + 376 \L^4 + 988 \L^5 + 1917 \L^6. 
\]
Using the Getzler--Kapranov formula and \cite{Gorsky}, we have $e(\Mb_6)=12218$. Thus, $\dim RH^{14}(\Mb_6)=2686$. In total, we have
\begin{equation*}
    \begin{split}
        \chi_c(\Mb_6)&=\L^{15}+5\L^{14}+25\L^{13}+111\L^{12}+376\L^{11}+988\L^{10}+1917\L^9+2686\L^8+2686\L^7+\\&1917\L^6+988\L^5+376\L^4+111\L^3+25\L^2+5\L+1.
    \end{split}
\end{equation*}

\end{example}
\begin{example}\label{ex:Mb61}
We calculate $\chi_c(\Mb_{6,1})$. The method is similar to Examples \ref{ex:Mb52} and \ref{ex:Mb53}. Using \texttt{admcycles}, we calculate
\[\sum_{i=0}^{6} \dim RH^{2i}(\Mb_{6,1}) \L^i = 1 + 8 \L + 60 \L^2 + 333 \L^3 + 1426 \L^4 + 4530 \L^5 + 10597 \L^6. \]
Then, we obtain two linear relations among the coefficients from the Euler characteristic $e(\Mb_{6,1})=90640$, see \cite{Gorsky}, and from the point count $\#\Mb_{6,1}(\F_2)=47655305$, which was obtained by the Getzler--Kapranov method and Huang, Kedlaya, and Lau's calculation $\#M_{6,1}(\F_2)=223317$ \cite{KHL}. As a result, we obtain $\dim RH^{14}(\Mb_{6,1})=17775$ and $\dim RH^{16}(\Mb_{6,1})=21180$. In total, we have
\begin{equation*}
    \begin{split}
\chi_c(\Mb_{6,1})&=\L^{16}+8\L^{15}+60\L^{14}+333\L^{13}+1426\L^{12}+4530\L^{11}+10597\L^{10}+17775\L^9+\\&21180\L
^8+17775\L^7+10597\L^6+4530\L^5+1426\L^4+333\L^3+60\L^2+8\L+1.
    \end{split}
\end{equation*}

\end{example}
The calculations in Examples \ref{ex:Mb5}--\ref{ex:Mb61} finish the proof of Theorem \ref{thm:Mbargn} when $g\geq 5$. \qedsymbol 



\begin{rem}
    The difficulty of the calculation using \texttt{admcycles} of $\dim RH^{2i}(\Mb_{g,n})$ increases rapidly with $i$, $g$, and $n$. For example, the calculations in Example \ref{ex:Mb5} take about an hour at most, whereas the higher codimension calculations in Examples \ref{ex:Mb53} and \ref{ex:Mb61} can take on the order of months. In addition to the issue of time, there is also the issue of computer memory. The matrices of Pixton's relations used to obtain upper bounds on $\dim RH^{2i}(\Mb_{g,n})$ are large, and the linear algebra computations can lead to fill-in, thus causing memory usage to increase rapidly. The pairing matrices computed to obtain lower bounds on $\dim RH^{2i}(\Mb_{g,n})$ are smaller, but they are dense, and thus also require significant memory usage.
\end{rem}

\section{Local systems in genus four}\label{sec:genus4}

In this section we explain the methods used when $g=4$. The reason that we apply different techniques in this case is that computations in \texttt{admcycles} do not scale well with the number of marked points, and the methods described in the preceding section are not computationally feasible by themselves, even though point counts over more finite fields are available when $g=4$ than in genus $5$ and $6$.

\subsection{Passing back and forth between local systems and $S_n$-equivariant Euler characteristics}\label{subsec: local systems}
Fix $g \geq 2$ for the rest of this subsection. For a partition $\lambda = (\lambda_1 \geq \lambda_2 \geq \dots \geq \lambda_g\geq 0)$, let $\V_\lambda$ denote the local system on $M_g$ associated to the corresponding representation of $\mathrm{Sp}_{2g}$. The $\V_\lambda$ are smooth $\ell$-adic sheaves of weight $|\lambda|= \sum_i \lambda_i$. If $\lambda$ is a partition of length greater than $g$, then we take $\V_\lambda$ to be zero. The local system $\V_{(1,0,0,\dots)}$ associated to the defining representation will be denoted $\V$. We write $S^\lambda(\V)$ for the Schur functor $S^\lambda$ applied to the local system $\V$. If $\mathbb W$ is any of the above local systems, then we set
\[ \chi_c(M_g,\mathbb W) := \sum_i (-1)^i [H^i_c(M_g,\mathbb W)] \in K_0(\Gal_\Q).\]
What will be important for us is that we can freely pass back and forth between the following two sets of data:
 \begin{equation}\label{eq:correspondence}
     \{\langle s_\lambda,\chi_c^{S_n}(M_{g,n})\rangle\} \longleftrightarrow \{ \chi_c(M_g,\V_{\lambda'})\}.
 \end{equation} 
In more detail, consider the following four functions from integer partitions to $K_0(\Gal_\Q)$:
\begin{align*}
	\lambda & \longmapsto \langle s_\lambda, \chi_c^{S_n}(M_{g,n})\rangle, \\
	\lambda & \longmapsto \langle s_\lambda, \chi_c^{S_n}(C_g^n)\rangle, \\
	\lambda & \longmapsto (-1)^{|\lambda|}\chi_c(M_g,S^{\lambda'}(\V)), \\
	\lambda & \longmapsto (-1)^{|\lambda|}\chi_c(M_g,\V_{\lambda'}).
\end{align*}
Here $C_g^n$ denotes the $n$-fold fibered power of the universal curve over $M_g$, and $\lambda'$ denotes the conjugate partition of $\lambda$. We claim that if $F$ and $G$ denote any two of these four functions, then there are matrix coefficients $a_{\lambda\mu} \in \Z[\L]$, given by explicit formulae, such that 
\[F(\lambda) = \sum_{\mu} a_{\lambda\mu} G(\mu)\]
for all $\lambda$; moreover, $a_{\lambda\mu}$ is nonzero only for $|\mu| < |\lambda|$, or if $\mu=\lambda$, in which case $a_{\lambda\lambda}=1$. In particular, if we fix $N$ and suppose that know any one of these functions for all $|\lambda| \leq N$, then we can determine the other three functions, too. This idea goes back to Getzler \cite{getzler-resolving}. 

Let us briefly justify the claim. 
For the first two, use that $C_g^n$ has a stratification with one open stratum $M_{g,n}$, $\binom n 2$ diagonal loci isomorphic to $M_{g,n-1}$, $\binom n 3 + \binom n {2,2,n-2}$ loci isomorphic to $M_{g,n-2}$, etc. Then additivity of Euler characteristics over stratifications gives
\[ \chi_c^{S_n}(C_g^n) = \chi_c^{S_n}(M_{g,n}) + (\text{terms involving $\chi_c^{S_{k}} (M_{g,k})$ for $k < n$} ). \]
For example, the first of these lower order terms is $\operatorname{Ind}_{S_{n-2}\times S_2}^{S_n} \operatorname{Res}^{S_{n-1}}_{S_{n-2}} \chi_c^{S_{n-1}}(M_{g,n-1})$, which expresses that along the diagonal locus where $x_i=x_j$ we have a configuration of $n-1$ points, one of which is singled out as special (it is the unique point of multiplicity two). In this case the matrix coefficients $a_{\lambda\mu}$ are in fact natural numbers. 

To pass between the second and the third, use that $H^*_c(C_g^n) = \mathbb H^*_c(M_g,(R\pi_*\Q_\ell)^{\otimes n})$, where $\pi: C\to M_g$ is the universal curve. We then decompose $(R\pi_*\Q_\ell)^{\otimes n}$ in the Grothendieck group of $S_n$-equivariant local systems on $M_g$, and then apply 
\[ \chi_c(M_g,-) : K_0(\mathrm{LSyst}_{M_g}^{S_n})\longrightarrow K_0(\mathrm{Gal}_\Q^{S_n}).\]
Schur--Weyl duality says that $V^{\otimes k} \cong \bigoplus_{|\lambda|=k} S^\lambda(V) \otimes \sigma_\lambda$, for any $k$ and any vector space $V$ in characteristic zero. Taking into account the K\"unneth theorem, and that the Koszul sign rule introduces a twist of $\mathrm{sgn}_k$ in the terms $(R^1 \pi_*\Q_\ell)^{\otimes k}$, we find
\[ [ (R\pi_*\Q_\ell)^{\otimes n}] = \sum_{r+s+t=n} h_r \cdot \L^t h_t \cdot (-1)^s\sum_{|\lambda|=s} [S^\lambda(\V)] s_{\lambda'} \in K_0(\mathrm{LSyst}_{M_g}) \otimes \Lambda_n.\]
The terms in the above sum correspond to K\"unneth components in $(R\pi_*\Q_\ell)^{\otimes n}$ with $r$ copies of $R^0\pi_*\Q_\ell \cong \Q_\ell$, $s$ copies of $R^1\pi_*\Q_\ell \cong \V$ and $t$ copies of $R^2\pi_*\Q_\ell \cong \L$.

The passage between the third and the fourth is the most classical. We can decompose each $S^\lambda(\V)$ as a direct sum of symplectic local systems $\V_\mu$, and the multiplicities are given precisely by the branching rule from $\mathrm{GL}_{2g}$ to $\mathrm{Sp}_{2g}$. 
The consequence of the branching rule is that 
\[ S^\lambda(\V) = \V_\lambda + (\text{lower order terms}),\]
with the lower order terms being of the form $\L^{(|\lambda|-|\mu|)/2}\V_\mu$, with $|\mu|<|\lambda|$.

\subsection{Weight bounds on motivic Euler characteristics of local systems}
We have already explained a computational advantage of $\chi_c(\overline M_{g,n})$ over $\chi_c(M_{g,n})$: 
if it is a polynomial in $\L$ then it is palindromic, by Poincar\'e duality, which effectively means that we only need to determine $\lceil \tfrac{3g-3+n} 2 \rceil$ unknown coefficients, instead of $3g-3+n$ coefficients. We now want to explain that the Deligne bounds \cite{weil2} on Frobenius eigenvalues imply another, independent, computational advantage of $\chi_c(M_g,\V_{\lambda})$: its high weight contributions are always zero. 

\begin{thm}[Deligne bounds]
	Let $X$ be an algebraic variety, $\mathbb V$ an $\ell$-adic sheaf on $X$ which is mixed of weights $\leq w$. Then $H^k_c(X,\V)$ is mixed of weights $\leq k+w$.   
\end{thm}

\begin{cor} \label{Deligne corollary}
The compactly supported cohomology $H^\bullet_c(M_g,\V_\lambda)$ is of weights $\leq 6g-6+|\lambda|$. In particular, if $\chi_c(M_g,\V_\lambda)$ is a polynomial in $\L$, then it is of degree at most $3g-3+\tfrac 1 2 |\lambda|$. 
	\end{cor}

Note that when it is polynomial in $\L$, the motivic Euler characteristic $\chi_c(\Mb_{g,n})$ is of degree $3g-3+n$. Since $n$ corresponds to $|\lambda|$ under the correspondences \eqref{eq:correspondenceGK} and \eqref{eq:correspondence}, Corollary \ref{Deligne corollary} gives a significant computational advantage when $n$ is large, see further in Proposition~\ref{weight proposition}.

We use known results about $H^\bullet(M_g,\V_\lambda)$ in low degrees (equivalently, compact support cohomology in high degrees) to deduce a few more results about the high weight contributions to 
$\chi_c(M_g,\V_{\lambda})$.

\begin{prop}\label{prop:H0}
	We have $H^0(M_g,\V_\lambda)=0$ unless $\lambda=0$. 
\end{prop} 
\begin{proof}
	It is classical that the mapping class group surjects to $\mathrm{Sp}_{2g}(\Z)$.
\end{proof}

The cohomology in degree $1$ has been completely calculated by Hain \cite{hain}, building on work of Johnson \cite{johnson}:

\begin{thm}[Hain]\label{thm:Hain} Let $g \geq 3$. Then $H^1(M_g,\V_{111})\cong \L^2$. If $\lambda \neq (1,1,1)$, then $H^1(M_g,\V_{\lambda})=0$. 	
\end{thm}

\begin{cor}\label{weight corollary}
    Let $g \geq 3$ and $\lambda \neq 0$. 
    \begin{enumerate}
    \item If $\lambda = (1,1,1)$, then $\chi_c(M_g,\V_\lambda)$ 
    has highest weight contribution 
    $-\L^{3g-2}$. 
    \item If $\lambda \neq (1,1,1)$, then $\chi_c(M_g,\V_\lambda)$ has contribution of weight at most $6g-8+ |\lambda|$. 
    \end{enumerate}
\end{cor}
\begin{proof}
    The result follows from Corollary \ref{Deligne corollary}, Proposition \ref{prop:H0}, Theorem \ref{thm:Hain}, and Poincar\'e duality.
\end{proof}

\begin{prop} \label{weight proposition}
Fix $(g,n)$ with $g \geq 3$. The following pieces of data determine $\chi_c^{S_n}(\Mb_{g,n})$.
\begin{enumerate}
    \item $\chi_c^{S_n}(\partial  M_{g,n})=\chi_c^{S_n}(\overline M_{g,n})-\chi_c^{S_n}(M_{g,n})$;
    \item $\chi_c^{S_n}(M_{g,n'})$ for all $n'<n$;
    \item for every partition $\lambda$ with $|\lambda|=n$, all contributions to $\chi_c(M_g,\V_{\lambda})$ of weights $w$ such that $3g-3+n \leq w \leq 6g-8+n$.
\end{enumerate}
\end{prop} 
\begin{rem} The statement is formulated for $g \geq 3$ only because for $g \leq 2$, $\chi_c^{S_n}(M_{g,n})$ is computable for every $n$; see \cite{GetzlerGenus0} for genus $0$, \cite{getzler-resolving} for genus $1$ and \cite{PetersenA2} together with \cite{Petersendelliptic} for genus $2$. For a few more details on the genus $2$ computation, see \cite[Section 4]{BergstromFaber}.
\end{rem}
\begin{rem} Note that the range of weights of the contributions to $\chi_c(M_g,\V_{\lambda})$, for $\lambda$ with $|\lambda|=n$, that are needed in the assumption of in Proposition~\ref{weight proposition}, is equal to $3g-4$  and thus does not depend on $n$. 
\end{rem}
\begin{proof} Since $\overline M_{g,n}$ is smooth and proper, if we know all the contributions of weight greater or equal to $\dim \overline M_{g,n}=3g-3+n$, then we can determine all contributions of smaller weight by Poincar\'e duality. Since we know $\chi_c^{S_n}(\partial  M_{g,n})$, we only need to determine the contributions to $\chi_c^{S_n}(M_{g,n})$ of weight greater or equal to $3g-3+n$.

As described in Section~\ref{subsec: local systems}, $\chi_c^{S_n}(M_{g,n})$ can be computed from $\chi_c(M_g,\V_{\lambda})$ for all $\lambda$ such that $|\lambda| \leq n$. From knowing $\chi_c^{S_n}(M_{g,n'})$ for all $n'<n$ we can compute $\chi_c(M_g,\V_{\lambda})$ for all  $|\lambda| < n$. When writing $\chi_c^{S_n}(M_{g,n})$ in terms of $\chi_c(M_g,\V_{\lambda})$ for $|\lambda| \leq n$, then all contributions with $|\lambda|=n$ appear without Tate twist (in fact with coefficient $(-1)^{|\lambda|}$ to $s_{\lambda'}$, where $\lambda'$ denotes the conjugate of $\lambda$, see Section~\ref{subsec: local systems}). 
The result now follows from Corollary~\ref{Deligne corollary}, Proposition~\ref{prop:H0} and Corollary~\ref{weight corollary} (where we use part $(2)$ in the case of $n=3$). 
\end{proof}

\subsection{Genus four}\label{sec:genus4final} Using the above general discussion, we now prove Theorems \ref{thm:Mgn} and \ref{thm:Mbargn} when $g=4$. The proof uses Proposition \ref{weight proposition}, which determines $\chi_c^{S_n}(\Mb_{g,n})$. But note that the input data includes $\chi_c^{S_n}(\partial M_{g,n})$, and knowing both $\chi_c^{S_n}(\Mb_{g,n})$ and $\chi_c^{S_n}(\partial M_{g,n})$ determines $\chi_c^{S_n}(M_{g,n})$. Alternatively, one can simply follow the proof of Theorem \ref{thm:Mgn} assuming Theorem \ref{thm:Mbargn} given in Section \ref{sec:GetzlerKapranov}.

\begin{proof}[Proof of Theorems \ref{thm:Mgn} and \ref{thm:Mbargn} when $g=4$]
Note first that $\chi_c^{S_n}(M_{1,1},\V_a)$ is known for any $a \geq 0$ by \cite{getzler-resolving} and hence that we can compute $\chi_c^{S_n}(M_{1,n})$ for any $n \geq 1$. So, if we fix $n \leq 7$ and use the results mentioned in Remark~\ref{smaller Eulers} and the Getzler-Kapranov method in Section~\ref{sec:GetzlerKapranov}, then if we know $\chi_c^{S_n}(M_{4,n'})$ for $n'<n$, we can compute $\chi_c^{S_n}(\partial  M_{4,n})$. Already from earlier, we know $\chi_c^{S_n}(M_{4,n})$ for $n \leq 3$, see \cite{BergstromFaberPayne}. 

By Theorem~\ref{thm:CLPW} we know that $\chi_c^{S_n}(M_{4,n})$ is Tate for $n \leq 6$. This means that $\chi_c(M_4,\V_\lambda)=\sum_i a_{i,\lambda} \L^i$ for some integers $a_{i,\lambda}$ for all $|\lambda| \leq 6$. By, Proposition~\ref{weight proposition} it follows that, for a fixed $n$, if we know $\chi_c^{S_n}(M_{4,n'})$ for $n'<n$ then we only need to determine $a_{i+\lceil (9+n)/2 \rceil,\lambda}$ for $i=0,1,2,3$ and all $\lambda$ such that $|\lambda|=n$ to be able to compute $\chi_c^{S_n}(M_{4,n})$. 

The $S_n$-equivariant integer valued Euler characteristic of $M_{4,n}$ is known for all $n$ (see \cite{Gorsky}), and hence $\sum_i a_{i,\lambda}$ is known for all $\lambda$ such that $|\lambda| \leq 6$. Furthermore, $\sum_i a_{i,\lambda} q^i$ is known from counts of points over finite fields for $q=2,3,4$ for all $\lambda$ such that $|\lambda| \leq 6$, see \cite{BergstromFaberPayne} for $q=2,3$ and from personal communication with Carel Faber for $q=4$.

\renewcommand{\arraystretch}{1.2}
\begin{table}
    
{\centering\footnotesize
\begin{tabularx}{\textwidth}{|c|>{\raggedright\arraybackslash}X|}		\hline
		$n $&$\chi_c^{S_n}(M_{4,n}) $\\
		\hline
		$4$&$
		\L^{13}+2\L^{12}+3\L^{11}-3\L^{10}-6\L^9+3\L^8-2\L^4+5\L^3-3\L)s_{4}$\\
		&$+(\L^{12}+2\L^{11}-\L^{10}-7\L^9+5\L^8-\L^7+\L^5-5\L^4+8\L^3-4\L-1)s_{3,1} $\\
		&$+
		(\L^{11}-\L^{10}-2\L^9+2\L^8-\L^7+\L^6+\L^5-3\L^4-3\L^2-3\L)s_{2^2}$\\
		&$+(-2\L^9+2\L^8-\L^5-5\L^4+\L^3+2\L^2+\L)s_{2,1^2}+
		(-\L^5-\L^4+2\L^2+3\L+1)s_{1^4} $\\
		 $5$&$
		(\L^{14}+2\L^{13}+3\L^{12}-3\L^{11}-8\L^{10}+4\L^9+2\L^8-\L^7+\L^6-2\L^5+4\L^4-2\L^3-4\L^2+
		3\L+2)s_{5}$\\
		&$+(\L^{13}+2\L^{12}-2\L^{11}-10\L^{10}+7\L^9+2\L^8-3\L^7+1+\L^6-3\L^5+9\L^4-6
		\L^2-\L)s_{4,1}$\\
		&$+(\L^{12}-\L^{11}-5\L^{10}+8\L^9-2\L^8-2\L^7+3\L^6+5\L^4-8\L^3-5\L^2+\L+1)
		s_{3,2}$\\
		&$+(-3\L^{10}-3+4\L^9-3\L^6-\L^5+4\L^4+3\L^3+5\L^2-6\L)s_{3,1^2}$\\
		&$+(-\L^{10}+3\L^
		9-\L^8+\L^7-2\L^6-\L^5-3\L^4-9\L^3+3\L^2+2\L)s_{2^2,1}$\\
		&$+(\L^9-5\L^6-\L^4+\L^3+11\L^2+
		2\L-1)s_{2,1,1,1}+(-\L^6+5\L^2+5\L+1)s_{1^5}
		$\\
		 $6$&$
		(\L^{15}+2\L^{14}+3\L^{13}-3\L^{12}-9\L^{11}+3\L^{10}+5\L^9-2\L^8+4\L^7-2\L^6-8\L^3+7\L^2+9
		\L)s_{6}$\\
		&$+(\L^{14}+2\L^{13}-2\L^{12}-12\L^{11}+9\L^{10}+9\L^9-12\L^8+7\L^7-3\L^6-2\L^5+11
		\L^4-14\L^3-4\L^2+14\L 
		+4)s_{5,1}$\\
		&$+(\L^{13}-2\L^{12}-8\L^{11}+11\L^{10}+4\L^9-13\L^8+10\L^
		7-9\L^5+11\L^4-9\L^3-8\L^2+10\L+1)s_{4,2}$\\
		&$+(-4\L^{11}+5\L^{10}+2\L^9-8\L^8+\L^7+5\L^
		6-8\L^5+24\L^4+10\L^3-36\L^2-9\L+1)s_{4,1^2}$\\
		&$+(-3\L^{11}+7\L^{10}+\L^9-9\L^8+6\L^7+5
		\L^6-2\L^5+3\L^4-9\L^3-10\L^2+6\L+2)s_{3^2}$\\
		&$+(-2\L^{11}+7\L^{10}-\L^9-7\L^8+6\L^7+6
		\L^6-22\L^5+13\L^4+16\L^3-27\L^2-3\L+2)s_{3,2,1}$\\
		&$+(\L^{10}-7\L^7+5\L^6-11\L^5+16\L^
		4+31\L^3-24\L^2-18\L-2)s_{3,1^3}$\\
		&$+(\L^{10}-2\L^9+\L^8+2\L^7-10\L^5-3\L^4+7\L^3-\L-\
		2)s_{2,2,2}$\\
		&$+(\L^{10}-4\L^7+2\L^6-14\L^5+3\L^4+18\L^3-9\L^2-\L+1)s_{2^2,1^2}$\\
		&$+(-3\L^7+3\L^6-4\L^5+4\L^4+18\L^3+5\L^2-1)s_{2,1^4}+(2\L^4+4\L^3+6\L^2+5\L+1)s_{1^6} $\\
		 $7$&$		(\L^{16}+2\L^{15}+3\L^{14}-3\L^{13}-10\L^{12}+2\L^{11}+7\L^{10}-4\L^9+7\L^8 -2\L^7-2\L^6
		+4\L^5-15\L^4+11\L^3+16\L^2-7\L-4)s_{7}$\\&$
		+(\L^{15}+2\L^{14}-2\L^{13}-13\L^{12}+8\L^{11}+16\L^{10}
		-20\L^9+17\L^8-8\L^7-7\L^6+29\L^5-49\L^4+4\L^3
		+54\L^2+5\L-3)s_{6,1}$\\&$
		+(\L^{14}-2\L^{13}
		-10\L^{12}+14\L^{11}+15\L^{10}-34\L^9+27\L^8-14\L^7-19\L^6+53\L^5-69\L^4+74\L^2
		-5\L-7)s_{5,2}$\\&$
		+(-4\L^{12}+6\L^{11}+8\L^{10}-22\L^9+14\L^8-7\L^7-7\L^6+64\L^5-49\L^4-57\L^3
		+44\L^2+28\L+6)s_{5,1^2}$\\&$
		+(-\L^{13}-6\L^{12}+10\L^{11}+10\L^{10}-28\L^9+24\L^8-6\L^7-12\L^6
		+47\L^5-68\L^4-19\L^3+61\L^2
		+\L-5)s_{4,3}$\\&$
		+(-3\L^{12}+9\L^{11}+4\L^{10}-33\L^9+37\L^8
		-16\L^7-31\L^6+103\L^5-81\L^4-63\L^3+82\L^2+5\L
		-3)s_{4,2,1}$\\&$
		+(\L^{11}-5\L^9+\L^8-4\L^7
		-6\L^6+59\L^5-10\L^4-68\L^3-3\L^2+2\L+3)s_{4,1^3}$\\&$
		+(-\L^{12}+5\L^{11}+2\L^{10}-18\L^9
		+23\L^8-2\L^7-17\L^6+51\L^5-58\L^4-41\L^3+64\L^2+20\L+2)s_{3^2,1}$\\&$
		+(3\L^{11}-2\L^{10}
		-11\L^9+22\L^8-10\L^7-23\L^6+41\L^5-34\L^4-5\L^3+44\L^2-20\L-9)s_{3,2^2}$\\&$
		+(\L^{11}
		-5\L^9+11\L^8-12\L^7-27\L^6+69\L^5-32\L^4-56\L^3+53\L^2+8\L)s_{3,2,1^2}$\\&$
		+(-3\L^8
		-\L^7-5\L^6+28\L^5+8\L^4-9\L^3+3\L^2-21\L-6)s_{3,1^4}$\\&$
		+(-\L^{10}+6\L^8-8\L^7-15\L^6
		+17\L^5-16\L^4-4\L^3+30\L^2-7\L-6)s_{2^3,1}$\\&$
		+(-\L^8-9\L^6+16\L^5-5\L^4-3\L^3
		+37\L^2+8\L-1)s_{2^2,1^3}$\\&$
		+(\L^7-2\L^6+8\L^5+3\L^4+12\L^3+21\L^2-3\L-4+S[12])s_{2,1^5}$\\&$
		+(2\L^5+\L^4+3\L^3+9\L^2+5\L-\L S[12])s_{1^7} $\\ \hline
	\end{tabularx}}\hspace{5pt}
    \caption{Euler characteristics of $M_{4,n}$.}
\end{table}

Following the proof of Proposition~\ref{weight proposition} we can write a formula for $\chi_c^{S_n}(\overline M_{4,4})$ in terms of the unknowns $a_{i+\lceil (3g-3+n)/2 \rceil,\lambda}$ for $i=0,1,2,3$ and $\lambda$ such that $|\lambda|=4$. Since we know $\chi_c^{S_n}(\partial  M_{4,4})$ by the above, we get $4$ relations (from the $q=1,2,3,4$ computations) among the $a_{i+\lceil (3g-3+n)/2 \rceil,\lambda}$ for $i=0,1,2,3$ and each $\lambda$ such that $|\lambda|=4$. These relations are independent, and hence we can determine $\chi_c^{S_n}(M_{4,4})$.

\begin{table}
	
	{\centering\footnotesize
		\begin{tabularx}{\textwidth}{|c|>{\raggedright\arraybackslash}X|}		\hline
			$n$&$\chi_c^{S_n}(\overline{M}_{4,n}) $ \\
			\hline
			 $4$&$ 
(\L^{13}+14\L^{12}+127\L^{11}+746\L^{10}+2842\L^9+6993\L^8+11002\L^7+\ldots
)s_{4}$\\&$+(7\L^{12}+108\L^{11}+800
\L^{10}+3472\L^9+9155\L^8+14863\L^7+\ldots
)s_{3,1}$\\&$
+(3\L^{12}+45\L^{11}+358\L^{10}+1606\L^9+4333\L^8+7102\L^7+\ldots
)s_{2^2}$\\&$
+(16\L^{11}+199\L^{10}+1070\L^9+3138\L^8+5329\L^7+\ldots
)s_{1^4}$\\
$5$&$
(\L^{14}+16\L^{13}+174\L^{12}+1188\L^{11}+5376\L^{10}+15870\L^9+30598\L^
8+38075\L^7+\ldots
)s_{5}$\\&$
+(10\L^{13}+175\L^{12}+1551\L^{11}+8079\L^{10}+26030\L^9+52457\L^8+66280\L^7+\ldots
)s_{4,1}$\\&$
+(5\L^{13}+109\L^{12}+1076\L^{11}+6020\L^{10}+20190\L^9+41612\L^8+52923\L^7+\ldots
)s_{3,2}$\\&$
+(34\L^{12}+523\L^{11}+3461\L^{10}+12695\L^9+27248\L^8+35136\L^7+\ldots
)s_{3,1^2}$\\&$
+(20\L^{12}+305\L^{11}+2105\L^{10}+7844\L^9+17060\L^8+22041\L^7+\ldots
)s_{2^2,1}$\\&$
+(38\L^{11}+422\L^{10}+1909\L^9+4537\L^8+6015\L^7+\ldots
)s_{2,1^3}$\\&$
+(8\L^{10}+52\L^9+152\L^8+207\L^7+\ldots
)s_{1^5}$\\
$6$ &$
(\L^{15}+19\L^{14}+229\L^{13}+1801\L^{12}+9399\L^{11}+32729\L^{10}+75657\L
^9+115311\L^8+\ldots
)s_{6}$\\&$
+(12\L^{14}+255\L^{13}+2637\L^{12}+16242\L^{11}+62594\L^{10}+153648\L^9+240821\L^8+\ldots
)s_{5,1}$\\&$
+(8\L^{14}+199\L^{13}+2347\L^{12}+15693\L^{11}+63872\L^{10}+161960\L^9+257796\L^8+\ldots
)s_{4,2}$\\&$
+(61\L^{13}+1094\L^{12}+8730\L^{11}+38926\L^{10}+103605\L^9+168420\L^8+\ldots
)s_{4,1^2}$\\&$
+(2\L^{14}+70\L^{13}+899\L^{12}+6350\L^{11}+26634\L^{10}+68765\L^9+110335\L^8+\ldots
)s_{3^2}$\\&$
+(52\L^{13}+1023\L^{12}+8726\L^{11}+40514\L^{10}+110380\L^9+181409\L^8+\ldots
)s_{3,2,1}$\\&$
+(115\L^{12}+1540\L^{11}+8722\L^{10}+26161\L^9+44781\L^8+\ldots
)s_{3,1^3}$\\&$
+(8\L^{13}+174\L^{12}+1588\L^{11}+7718\L^{10}+21585\L^9+35906\L^8+\ldots
)s_{2^3}$\\&$
+(74\L^{12}+1052\L^{11}+6055\L^{10}+18382\L^9+31627\L^8+\ldots
)s_{2^2,1^2}$\\&$
+(61\L^{11}+559\L^{10}+2058\L^9+3824\L^8+\ldots
)s_{2,1^4}$\\&$
+(4\L^{10}+25\L^9+52\L^8+\ldots
)s_{1^6}$\\
$7$&$
(\L^{16}+21\L^{15}+291\L^{14}+2575\L^{13}+15378\L^{12}+61966\L^{11}+169004\L^{10}
+309547\L^9+379168\L^8+\ldots
)s_{7}$\\&$
+(15\L^{15}+352\L^{14}+4171\L^{13}+29627\L^{12}+133899\L^{11}+392209\L^{10}+748019\L^9+927874\L^8+\ldots
)s_{6,1}$\\&$
+(10\L^{15}+308\L^{14}+4250\L^{13}+33501\L^{12}+162093\L^{11}+496352\L^{10}+970380\L^9+1213554\L^8+\ldots
)s_{5,2}$\\&$
+(94\L^{14}+1975\L^{13}+18528\L^{12}+98592\L^{11}+318089\L^{10}
+639173\L^9+805894\L^8+\ldots
)s_{5,1^2}$\\&$
+(5\L^{15}+175\L^{14}+2646\L^{13}+22098\L^{12}+111285\L^{11}+349514\L^{10}+693261\L^9+870955\L^8+\ldots
)s_{4,3}$\\&$
+(102\L^{14}+2377\L^{13}+24255\L^{12}+136140\L^{11}+454840\L^{10}+931456\L^9+1181947\L^8+\ldots
)s_{4,2,1}$\\&$
+(257\L^{13}+4186\L^{12}+28810\L^{11}+106942\L^{10}+230613\L^9
+297256\L^8+\ldots
)s_{4,1^3}$\\&$
+(38\L^{14}+1003\L^{13}+10758\L^{12}+62435\L^{11}+212531\L^{10}+439900\L^9+559919\L^8+\ldots
)s_{3^2,1}$\\&$
+(22\L^{14}+626\L^{13}+7222\L^{12}+43606\L^{11}+152264\L^{10}+319202\L^9+408131\L^8+\ldots
)s_{3,2^2}$\\&$
+(278\L^{13}+4791\L^{12}+34291\L^{11}+130223\L^{10}+284439\L^9+368085\L^8+\ldots
)s_{3,2,1^2}$\\&$
+(254\L^{12}+2904\L^{11}+13557\L^{10}+32461\L^9+43208\L^8+\ldots
)s_{3,1^4}$\\&$
+(66\L^{13}+1194\L^{12}+8904\L^{11}+34636\L^{10}+76689\L^9+99661\L^8+\ldots
)s_{2^3,1}$\\&$
+(183\L^{12}+2177\L^{11}+10216\L^{10}+24668\L^9+32841\L^8+\ldots
)s_{2^2,1^3}$\\&$
+(61\L^{11}+449\L^{10}+1298\L^9+1813\L^8+\ldots
)s_{2,1^5}$\\&$
+(\L^{10}+5\L^9+6\L^8+\ldots
)s_{1^7}$\\  \hline
	\end{tabularx}}\hspace{5pt}
	\caption{Euler characteristics of $\overline{M}_{4,n}$.}
\end{table}

In the same way, we can now use the information above to get $4$ relations from which we can determine $\chi_c^{S_n}(M_{4,5})$, and then in the next step $\chi_c^{S_n}(M_{4,6})$. 

Computing $\chi_c^{S_n}(\partial  M_{4,7})$ we find the contribution $S[12]s_{2,1^5}-\L S[12]s_{1^7}$, and all other contributions are of Tate type. The contributions from $\chi_c^{S_n}(M_{4,7})$ that are not determined by duality as above are the ones of weight between $16$ and $23$. It will then suffice to show that $H^k(\Mb_{4,7})$ is of Tate type (between degree $16$ and $23$) because then we can use the $4$ relations from $q=1,2,3,4$ in the same way as above to determine $\chi_c^{S_n}(M_{4,7})$.

\begin{table}
	
	{\centering\footnotesize
		\begin{tabular}{|l|l|}
        \hline
			$\lambda $&$\chi_c(M_{4},\V_\lambda) $ \\
			\hline
			 ${4}$ & $-\L^5+\L^3, $ \\
			 ${3, 1}$ & $\L^8-2\L^5-\L^4+\L^3+\L,$ \\
			 ${2, 2}$ & $-\L^8+\L^7+\L^4-\L^2-2\L,$ \\
			 ${2, 1, 1}$ & $-\L^9+\L^7+2\L^3-\L-1,$ \\
			 ${1, 1, 1, 1}$ & $-\L^8-\L^2-3\L-1,$ \\
			 ${5}$ & $-\L^5+\L^3-3\L^2-4\L-1, $ \\
			 ${4, 1}$ & $2\L^6-2\L^5+2\L^4+\L^3-7\L^2-\L+1, $ \\
			 ${2, 1, 1, 1}$ & $\L^9-\L^8+\L^7-\L^5+2\L^4-2\L^3+2\L,$ \\
			 ${3, 1, 1}$ & $-2\L^8+\L^7+\L^6-2\L^5+4\L^4-2\L^3-4\L^2+4\L+2, $ \\
			 ${3, 2}$ & $\L^9-2\L^8+\L^6-2\L^5+5\L^4-5\L^2+\L+1, $ \\
			 ${2, 2, 1}$ & $2\L^9-2\L^8+\L^7-\L^6-2\L^5+5\L^4-2\L^3-\L^2+3\L+1, $ \\
			 ${6}$ & $-\L^5+2\L^4+2\L^3-\L^2, $ \\
			 ${5, 1}$ & $\L^6-2\L^5+5\L^4+6\L^3-\L^2-1, $ \\
			 ${4, 2}$ & $\L^7-3\L^6-4\L^5+7\L^4+6\L^3-5\L^2-5\L-2, $ \\
			 ${4, 1, 1}$ & $ -3\L^6-3\L^5+7\L^4+6\L^3-4\L^2-6\L-2, $ \\
			 ${3, 3}$ & $ 
			-\L^{10}-\L^9+\L^8+\L^7-\L^6-2\L^5+2\L^4+2\L^3-\L^2-2\L-1, $ \\
			 ${3, 2, 1}$ & $-\L^{10}+\L^9+2\L^8+2\L^7-4\L^6-4\L^5+8\L^4+4\L^3-7\L^2-6\L-1, $ \\
			 ${3, 1, 1, 1}$ & $\L^7-\L^6-\L^5+3\L^4+\L^3-8\L^2-6\L-2, $ \\
			 ${2, 2, 2}$ & $\L^4+\L^3-3\L^2, $ \\
			 ${2, 2, 1, 1}$ & $\L^{11}-\L^{10}+\L^8+\L^7-2\L^6-2\L^5+2\L^4-\L, $ \\
			 ${7}$ & $ -\L^5+2\L^4+\L^3-6\L^2-4\L+\L S[12],$ \\
			 ${6,1}$ & $ -\L^5+7\L^4-7\L^3-16\L^2+3\L+4-S[12],$ \\
			 ${5,2}$ & $ 2\L^7-4\L^6+5\L^5+9\L^4-22\L^3-17\L^2+6\L+3,$ \\
			 ${5,1,1}$ & $ 3\L^7-5\L^6+\L^5+5\L^4-19\L^3-6\L^2+14\L+5,$ \\
			 ${4,3}$ & $ \L^{10}-\L^9+\L^7-3\L^6+9\L^5+6\L^4-16\L^3-7\L^2+7\L+3,$ \\
			 ${4,2,1}$ & $ -2\L^9+2\L^8+4\L^7-10\L^6+16\L^5+7\L^4-33\L^3-2\L^2+16\L+4,$ \\
			 ${4,1,1,1}$ & $ -\L^9+\L^8+2\L^7-4\L^6+5\L^5-11\L^3+5\L^2+8\L+1,$ \\
			 ${3,3,1}$ & $ \L^{10}-3\L^9+2\L^8-\L^7-6\L^6+12\L^5+\L^4-15\L^3+\L^2+9\L+3,$ \\
			 ${3,2,2}$ & $ \L^{10}-3\L^9+2\L^8-2\L^7-4\L^6+12\L^5-\L^4-13\L^3-\L^2+3\L,$ \\
			 ${3,2,1,1}$ & $ \L^{10}-2\L^9+2\L^8-\L^7-5\L^6+12\L^5-\L^4-13\L^3+5\L^2+7\L+1,$ \\
			 ${2,2,2,1}$ & $ \L^{10}-\L^9-\L^7-\L^6+5\L^5-\L^3+2\L^2. $\\ \hline
	\end{tabular}}
	\caption{Euler characteristics of $M_{4}$ with local system coefficients.}
\end{table}

Consider the sequence 
\[
H^{k-2}(\widetilde{\partial M}_{4,7})(-1)\to H^k(\Mb_{4,7}) \to W_kH^k(M_{4,7})\to 0
\]
as in \cite[Section 2]{petersentommasi}. Since $M_{4,7}$ has the Chow--K\"unneth generation property we know that the weight $k$ contribution to $H^{k}(M_{4,7})$ is of Tate type, by \cite[Lemma 9.11]{CL} and \cite[Lemma 4.3]{CLP}. Since the boundary divisors will consist of contributions from the cohomology of $\Mb_{3,9}/S_2$, $\Mb_{4,1+j} \times \Mb_{0,8-j}$ for $0 \leq j \leq 5$, $\Mb_{4-i,1+j} \times \Mb_{i,8-j}$ for $1 \leq i \leq 2$ and $0\leq j \leq 7$, we know by Remark~\ref{smaller Eulers} and the previous computations, that the cohomology of $\widetilde{\partial M}_{4,7}$ will be of Tate type. 
From this we conclude that the cohomology of $\Mb_{4,7}$ is of Tate type.
\end{proof}

\bibliographystyle{alpha}
\bibliography{bibax}

\end{document}